\documentclass[a4paper,12pt]{article}
\usepackage{amsmath,amsthm,amssymb}

\numberwithin{equation}{section}

\newtheorem{theorem}{Theorem}
\newtheorem{corollary}{Corollary}
\newtheorem{lemma}{Lemma}

\theoremstyle{definition}
\newtheorem{remark}{Remark}

\newcommand{\R}{\mathbb{R}}
\newcommand{\N}{\mathbb{N}}
\newcommand{\Z}{\mathbb{Z}}
\newcommand{\dual}[2]{\langle #1, #2\rangle}
\newcommand{\BLO}{\mathcal{L}}
\newcommand{\Trans}{\mathcal{T}}
\newcommand{\odd}{\mathop{\mathrm{odd}}\nolimits}
\newcommand{\even}{\mathop{\mathrm{even}}\nolimits}
\newcommand{\rad}{\mathop{\mathrm{rad}}\nolimits}

\begin{document}

\begin{center}
\Large{\textbf{Instability of bound states for abstract 
nonlinear Schr\"{o}dinger equations}}
\end{center}

\vspace{5mm}

\begin{center}
{\large Masahito OHTA}\footnote{Permanent address: 
Department of Mathematics, Saitama University, Saitama 338-8570, Japan
({\tt mohta@mail.saitama-u.ac.jp})}
\end{center}
\begin{center}
Institut de Math\'ematiques de Bordeaux, Universit\'e Bordeaux 1, \\
351 cours de la lib\'eration, 33405 Talence Cedex, France
\end{center}

\begin{abstract}
We study the instability of bound states for abstract nonlinear Schr\"odinger equations. 
We prove a new instability result for a borderline case between stability and instability. 
We also reprove some known results in a unified way. 
\end{abstract}

\section{Introduction}\label{sect:intro}

Following a celebrated paper \cite{GSS1} by Grillakis, Shatah and Strauss, 
we consider abstract Hamiltonian systems of the form 
\begin{equation}\label{eq:1.1}
\frac{du}{dt}(t)=\tilde JE'(u(t)),
\end{equation}
where $E$ is the energy functional on a real Hilbert space $X$, 
$J$ is a skew-symmetric operator on $X$, 
and $\tilde J$ is a natural extension of $J$ to the dual space $X^*$. 
We assume that \eqref{eq:1.1} is invariant under a one-parameter group 
$\{\Trans (s)\}_{s\in \R}$ of unitary operators on $X$, 
and study the instability of bound states $\Trans (\omega t)\phi_{\omega}$, 
where $\omega\in \R$ and $\phi_{\omega}$ is a solution of 
the corresponding stationary problem. 
Precise formulation of the problem will be set up 
in Section \ref{sect:form} based on \cite{GSS1}. 
We also borrow some notation from \cite{GSS2}, 
Comech and Pelinovsky \cite{CP} and Stuart \cite{stu}. 
Although it is desirable to work on the same general framework 
as in \cite{GSS1}, we need stronger assumptions for our purpose 
which will be explained below. 
We will formulate our assumptions in order to apply our theorems 
to nonlinear Schr\"odinger equations. 
In particular, we assume that the group $\{\Trans (s)\}_{s\in \R}$ is generated by 
the skew-symmetric operator $J$, that $J$ is bijective from $X$ to itself, 
and that the charge functional $Q$ is positive definite. 
These assumptions exclude nonlinear Klein-Gordon equations 
and KdV type equations from our framework. 
Moreover, we introduce an intermediate space $H$ 
between the energy space $X$ and the dual space $X^*$, 
which is a symmetry-constrained $L^2$ space 
in application to nonlinear Schr\"odinger equations. 
Such space as $H$ does not appear in \cite{GSS1}, 
but it will make the description of the theory simpler. 

In Section \ref{sect:results}, we state two main Theorems and four Corollaries. 
In Theorem \ref{thm1} we give a general sufficient condition 
for instability of bound states in non-degenerate case. 
We clarify that the conditions (A1), (A2a) and (A3) are essential 
in the proof of the instability theorem of \cite{GSS1}. 
We note that Theorem \ref{thm1} is inspired by a recent paper \cite{mae2} of Maeda. 
In fact, the condition (A3) appears explicitly in \cite{mae2} but not in \cite{GSS1}. 
It would be interesting that Theorem \ref{thm1} unifies two different known results, 
Corollaries \ref{cor3} and \ref{cor4}. 
Here, Corollary \ref{cor3} is a classical result due to \cite{GSS1,SS1}, 
while Corollary \ref{cor4} is originally due to \cite{mae2} with modifications. 
Although the key Lemma \ref{lem3} for the proof of Theorem \ref{thm1} 
is the same as Lemma 4.4 of \cite{GSS1}, 
some improvements are made in the proof of Lemma \ref{lem3}. 
For example, the function $\Lambda (\cdot)$ in Lemma \ref{lem3} 
is directly given by \eqref{eq:4.2} in the present paper, 
while in \cite{GSS1} it is determined by solving a differential equation 
and by the implicit function theorem (see (4.6) and Lemma 4.3 of \cite{GSS1}). 
It should be also mentioned that the proof of Lemma \ref{lem3} relies only on 
some simple Taylor expansions as in the proof of the stability theorem 
(see Theorem 3.4 of \cite{GSS1} and \cite{wei2}). 

On the other hand, in Theorem \ref{thm2}, 
we study the instability of bound states in a degenerate or critical case. 
We give two corollaries of Theorem \ref{thm2}. 
Corollary \ref{cor1} is a special case of Theorem \ref{thm2}, 
but it is a new result and will be useful 
to study the instability of bound states at a bifurcation point. 
While, Corollary \ref{cor2} is originally due to Comech and Pelinovsky \cite{CP}. 
We notice that our proof is completely different from that of \cite{CP}. 
In fact, the proof of \cite{CP} is based on a careful analysis of the linearized system, 
while Theorem \ref{thm2} is based on the Lyapunov functional method 
as well as Theorem \ref{thm1}. 
Our proof may be simpler, at least shorter than that of \cite{CP}. 
Another advantage of our approach is that 
Corollary \ref{cor2} requires the minimal regularity $E\in C^3(X,\R)$, 
while a higher regularity of $E$ is needed in \cite{CP} in application 
to nonlinear Schr\"odinger equations, especially for higher dimensional case 
(see Assumption 2.10, Remark 2.11 and Appendix B of \cite{CP}). 
As stated above, our abstract theorems are not applicable to 
nonlinear Klein-Gordon equations. 
For an instability result on NLKG in a critical case, 
see Theorem 4 of \cite{OT}. 

In Section \ref{sect:pre}, we recall some basic lemmas proved by \cite{GSS1}, 
and the proofs of Theorems \ref{thm1} and \ref{thm2} are given 
in Sections \ref{sect:proofthm1} and \ref{sect:proofthm2}, respectively. 
The representation formula \eqref{eq:4.5} of functional $P$ 
plays an important role especially in the proof of Theorem \ref{thm2}. 
Corollaries \ref{cor2}--\ref{cor4} are proved in Section \ref{sect:proofcor}. 
In Section \ref{sect:examples}, we give three examples. 
In Subsection \ref{ss:1}, we consider a simple example 
to explain the role of the assumption (A3) in Theorem \ref{thm1}. 
In Subsection \ref{ss:delta}, we apply Corollaries \ref{cor2} and \ref{cor4} 
to a nonlinear Schr\"odinger equation with a delta function potential, 
and give some remarks to complement the previous results in \cite{FJ,FOO,LFF}. 
In Subsection \ref{ss:system}, we apply Theorem \ref{thm1} to 
a system of nonlinear Schr\"odinger equations, 
and also mention the applicabililty of Theorem \ref{thm2} and Corollary \ref{cor1} 
to the problem at the bifurcation point. 

\section{Formulation}\label{sect:form}

Let $X$ and $H$ be two real Hilbert spaces with dual spaces $X^*$ and $H^*$ such that 
$$X\hookrightarrow H\cong H^*\hookrightarrow X^*$$
with continuous and dense embeddings. 
We denote the inner product and the norm of $X$ by $(\cdot,\cdot)_X$ and $\|\cdot\|_X$, 
and those of $H$ by $(\cdot,\cdot)_H$ and $\|\cdot\|_H$. 
We identify $H$ with $H^*$ by the Riesz isomorphism 
$I:H\to H^*$ defined by $\dual{Iu}{v}=(u,v)_H$ for $u,v\in H$. 
Here and hereafter, $\dual{\cdot}{\cdot}$ denotes 
the pairing between a Banach space and its dual space. 
Let $R:X\to X^*$ be the Riesz isomorphism between $X$ and $X^*$ defined by 
$$\dual{Ru}{v}=(u,v)_X, \quad u,v\in X.$$ 
Let $J\in \BLO(X)$ be bijective and skew-symmetric in the sense that 
\begin{equation}\label{eq:2.1}
(Ju,v)_X=-(u,Jv)_X, \quad (Ju,v)_H=-(u,Jv)_H, \quad u,v\in X.
\end{equation}
The operator $J$ is naturally extended to $\tilde J:X^*\to X^*$ defined by 
$$\dual{\tilde Jf}{u}=-\dual{f}{Ju}, \quad u\in X,~ f\in X^*.$$
Let $\{\Trans (s)\}_{s\in \R}$ be the one-parameter group of unitary operators 
on $X$ generated by $J$. By \eqref{eq:2.1}, we have 
$$\|\Trans (s)u\|_X=\|u\|_X, \quad \|\Trans (s)u\|_H=\|u\|_H, \quad s\in \R,~ u\in X.$$ 
We assume that $\Trans$ is $2\pi$-periodic, that is, 
$\Trans (s+2\pi)=\Trans (s)$ for $s\in \R$. 
The operator $\Trans (s)$ is naturally extended to 
$\tilde \Trans (s):X^*\to X^*$ defined by 
\begin{equation}\label{eq:2.2}
\dual{\tilde \Trans (s)f}{u}=\dual{f}{\Trans (-s)u}, \quad u\in X,~ f\in X^*.
\end{equation}
Then, $\{\tilde \Trans (s)\}_{s\in \R}$ is the one-parameter group 
of unitary operators on $X^*$ generated by $\tilde J$. 
Let $E\in C^2(X,\R)$, and we consider the equation 
\begin{equation}\label{eq:2.3}
\frac{du}{dt}(t)=\tilde JE'(u(t)). 
\end{equation}
We say that $u(t)$ is a solution of \eqref{eq:2.3} in an interval $\mathcal{I}$ of $\R$ 
if $u\in C(\mathcal{I},X)\cap C^1(\mathcal{I},X^*)$ and 
satisfies \eqref{eq:2.3} in $X^*$ for all $t\in \mathcal{I}$. 
We assume that $E$ is invariant under $\Trans$, that is, 
$E(\Trans (s)u)=E(u)$ for $s\in \R$ and $u\in X$. Then 
\begin{equation}\label{eq:2.4}
E'(\Trans (s)u)=\tilde \Trans (s)E'(u), \quad s\in \R,~ u\in X.
\end{equation}
We define $Q:X\to \R$ by 
$$Q(u)=\frac{1}{2}\|u\|_H^2, \quad u\in X.$$
Then, $Q'(u)=Iu$ for $u\in X$, and 
\begin{equation}\label{eq:2.5}
Q(\Trans (s)u)=Q(u), \quad Q'(\Trans (s)u)=\tilde \Trans (s)Q'(u), 
\quad s\in \R, ~ u\in X.
\end{equation}
We assume that the Cauchy problem for \eqref{eq:2.3} 
is locally well-posed in $X$ in the following sense. 

\vspace{2mm} \noindent{\bf Assumption.} 
For each $u_0\in X$ there exists $t_0>0$ depending only on $k$, 
where $\|u_0\|_X\le k$, and there exists a unique solution $u(t)$ of 
\eqref{eq:2.3} in the interval $[0,t_0)$ such that $u(0)=u_0$ and 
$E(u(t))=E(u_0)$, $Q(u(t))=Q(u_0)$ for all $t\in [0,t_0)$. 
\vspace{2mm}

By a {\it bound state} we mean a solution of \eqref{eq:2.3} 
of the form $u(t)=\Trans (\omega t)\phi$, 
where $\omega \in \R$ and $\phi\in X$ satisfies $E'(\phi)=\omega Q'(\phi)$. 

\vspace{2mm} \noindent{\bf Definition.}
We say that a bound state $\Trans (\omega t)\phi$ of \eqref{eq:2.3}
is {\it stable} if for all $\varepsilon>0$ there exists $\delta>0$ 
with the following property. 
If $\|u_0-\phi\|_X<\delta$ and $u(t)$ is the solution of \eqref{eq:2.3} 
with $u(0)=u_0$, then $u(t)$ exists for all $t\ge 0$ and 
$u(t)\in \mathcal{N}_{\varepsilon}(\phi)$ for all $t\ge 0$, where 
$$\mathcal{N}_{\varepsilon}(\phi)
=\{u\in X: \inf_{s\in \R}\|u-\Trans (s)\phi\|_X<\varepsilon\}.$$
Otherwise $\Trans (\omega t)\phi$ is called {\it unstable}. 

\section{Main Results}\label{sect:results}

In Sections \ref{sect:results}--\ref{sect:proofcor}, 
we assume all the requirements in Section \ref{sect:form}. 
For $\omega\in \R$ we define $S_{\omega}:X\to \R$ by 
$S_{\omega}(u)=E(u)-\omega Q(u)$ for $u\in X$. 
To state our main results, we impose the following conditions. 

\vspace{2mm} \noindent {\bf (A1).} 
There exist $\omega\in \R$ and $\phi_{\omega}\in X$ such that 
$S_{\omega}'(\phi_{\omega})=0$, $\phi_{\omega}\ne 0$ 
and $R\phi_{\omega} \in I(X)$. 

\vspace{2mm} \noindent {\bf (A2a).} 
There exists $\psi\in X$ such that $\|\psi\|_H=1$, $(\phi_{\omega},\psi)_H=0$, 
$(J\phi_{\omega},\psi)_H=0$ and $\dual{S_{\omega}''(\phi_{\omega})\psi}{\psi}<0$.

\vspace{2mm} \noindent {\bf (A2b).} 
$E\in C^3(X,\R)$. There exist $\psi\in X$ and $\mu\in \R$ such that 
$\|\psi\|_H=1$, $(\phi_{\omega},\psi)_H=0$, 
$(J\phi_{\omega},\psi)_H=(J\phi_{\omega},\psi)_X=0$ and 
\begin{equation}\label{eq:3.1}
S_{\omega}''(\phi_{\omega})\psi=\mu Q'(\phi_{\omega}), \quad 
\dual{S_{\omega}'''(\phi_{\omega})(\psi,\psi)}{\psi}\ne 3\mu.
\end{equation}

\vspace{2mm} \noindent {\bf (A3).} 
There exists a constant $k_0>0$ such that
\begin{equation}\label{eq:3.2}
\dual{S_{\omega}''(\phi_{\omega})w}{w}\ge k_0 \|w\|_X^2
\end{equation}
for all $w\in X$ satisfying $(\phi_{\omega},w)_H=(J\phi_{\omega},w)_H=(\psi,w)_H=0$. 

\begin{remark}\label{rem1}
By \eqref{eq:2.4} and \eqref{eq:2.5}, we see that 
$S_{\omega}'(\Trans (s)\phi_{\omega})=0$ for all $s\in \R$, 
and that $S_{\omega}''(\phi_{\omega})(J\phi_{\omega})=0$. 
The condition $(J\phi_{\omega},\psi)_X=0$ is assumed in (A2b) but not in (A2a). 
\end{remark}

\begin{remark}\label{rem2}
By (A2b), we have 
$\dual{S_{\omega}''(\phi_{\omega})\psi}{\psi}=\mu (\phi_{\omega},\psi)_H=0$. 
Moreover, $\dual{S_{\omega}''(\phi_{\omega})\psi}{w}=0$ 
for all $w\in X$ satisfying $(w,\phi_{\omega})_H=0$. 
\end{remark}

The main results of this paper are the following. 

\begin{theorem}\label{thm1}
Assume $({\rm A1})$, $({\rm A2a})$ and $({\rm A3})$. 
Then the bound state $\Trans (\omega t)\phi_{\omega}$ is unstable. 
\end{theorem}

\begin{theorem}\label{thm2}
Assume $({\rm A1})$, $({\rm A2b})$ and $({\rm A3})$. 
Then the bound state $\Trans (\omega t)\phi_{\omega}$ is unstable. 
\end{theorem}

The following Corollary \ref{cor1} is a special case of Theorem \ref{thm2} 
such that $\mu=0$ in $({\rm A2b})$. When $\mu=0$ in $({\rm A2b})$, 
the kernel of $S_{\omega}''(\phi_{\omega})$ contains a nontrivial element $\psi$ 
other than $J\phi_{\omega}$ which comes from the symmetry (see Remark \ref{rem1}). 
This is a typical situation at a bifurcation point 
(see Case (ii) of Example D in Section 6 of \cite{GSS1} and \cite{KKSW}), 
and Corollary \ref{cor1} will be useful to study the instability 
of bound states at the bifurcation point (see Subsection \ref{ss:system}). 

\begin{corollary}\label{cor1}
Assume $({\rm A1})$ and $E\in C^3(X,\R)$. 
Assume further that there exists $\psi\in X\setminus\{0\}$ such that 
$(\phi_{\omega},\psi)_H=0$, $(J\phi_{\omega},\psi)_H=(J\phi_{\omega},\psi)_X=0$, 
and that the kernel of $S_{\omega}''(\phi_{\omega})$ is spanned 
by $J\phi_{\omega}$ and $\psi$. 
If $\dual{S_{\omega}'''(\phi_{\omega})(\psi,\psi)}{\psi}\ne 0$ and $({\rm A3})$ holds, 
then the bound state $\Trans (\omega t)\phi_{\omega}$ is unstable. 
\end{corollary}

Next, we show that some known results are obtained 
as corollaries of Theorems \ref{thm1} and \ref{thm2}. 
For this purpose, we impose the following conditions. 

\vspace{2mm} \noindent {\bf (B1).} 
There exist an open interval $\Omega$ of $\R$ and 
a mapping $\omega \mapsto \phi_{\omega}$ from $\Omega$ to $X$ which is $C^1$ 
such that for each $\omega \in \Omega$, 
$S_{\omega}'(\phi_{\omega})=0$, $\phi_{\omega}\ne 0$, $R\phi_{\omega} \in I(X)$ 
and $(J\phi_{\omega},\phi_{\omega}')_H=(J\phi_{\omega},\phi_{\omega}')_X=0$, 
where $\phi_{\omega}'=d\phi_{\omega}/d\omega$. 

\vspace{2mm} \noindent {\bf (B2a).} 
There exist a negative constant $\lambda_{\omega}<0$ and a vector $\chi_{\omega}\in X$ 
such that $S_{\omega}''(\phi_{\omega})\chi_{\omega}=\lambda_{\omega}I\chi_{\omega}$, 
$\|\chi_{\omega}\|_H=1$, and $\dual{S_{\omega}''(\phi_{\omega})p}{p}>0$ for all $p\in X$ 
satisfying $(\chi_{\omega},p)_H=(J\phi_{\omega},p)_H=0$ and $p\ne 0$. 

\vspace{2mm} \noindent {\bf (B2b).} 
There exist two negative constants $\lambda_{0,\omega}$, $\lambda_{1,\omega}<0$ 
and vectors $\chi_{0,\omega}$, $\chi_{1,\omega}\in X$ such that 
$(\chi_{0,\omega},\chi_{1,\omega})_H=(\chi_{1,\omega},\phi_{\omega})_H=0$, 
$$S_{\omega}''(\phi_{\omega})\chi_{j,\omega}=\lambda_{j,\omega}I\chi_{j,\omega}, \quad 
\|\chi_{j,\omega}\|_H=1 \qquad (j=0,1),$$
and $\dual{S_{\omega}''(\phi_{\omega})p}{p}>0$ for all $p\in X$ satisfying 
$(\chi_{0,\omega},p)_H=(\chi_{1,\omega},p)_H=(J\phi_{\omega},p)_H=0$ and $p\ne 0$. 

\vspace{2mm} \noindent {\bf (B3).} 
The functional $u\mapsto \dual{S_{\omega}''(\phi_{\omega})u}{u}$ 
is weakly lower semi-continuous on $X$, 
and there exist positive constants $C_1$ and $C_2$ such that 
\begin{equation}\label{eq:3.3}
C_1\|u\|_X^2\le \dual{S_{\omega}''(\phi_{\omega})u}{u}+C_2\|u\|_H^2
\end{equation}
for all $u\in X$. 
Moreover, if a sequence $(u_n)$ of $X$ satisfies $\|u_n\|_X=1$ for all $n\in \N$ 
and $u_n\rightharpoonup 0$ weakly in $X$, then 
$\liminf_{n\to \infty}\dual{S_{\omega}''(\phi_{\omega})u_n}{u_n}>0$. 
\vspace{2mm}

We define $d(\omega)=S_{\omega}(\phi_{\omega})$ for $\omega\in \Omega$. 
As a corollary of Theorem \ref{thm2}, we have the following result 
which was proved in \cite{CP} 
assuming a higher regularity of the energy functional $E$. 

\begin{corollary}\label{cor2}
Assume $({\rm B1})$ and that for each $\omega\in \Omega$, 
$({\rm B2a})$ and $({\rm B3})$ hold. 
Assume further that $E\in C^3(X,\R)$ and that 
$\omega \mapsto \phi_{\omega}$ is $C^2$ from $\Omega$ to $X$. 
If $\omega_0\in \Omega$ satisfies $d''(\omega_0)=0$ and $d'''(\omega_0)\ne 0$, 
then the bound state $\Trans (\omega_0 t)\phi_{\omega_0}$ is unstable. 
\end{corollary}

On the other hand, as corollaries of Theorem \ref{thm1}, we have the following results. 
Corollary \ref{cor3} is a classical result due to \cite{GSS1,SS1}, 
while Corollary \ref{cor4} is an abstract generalization of the result in \cite{mae2}. 

\begin{corollary}\label{cor3}
Assume $({\rm B1})$ and that for each $\omega\in \Omega$, 
$({\rm B2a})$ and $({\rm B3})$ hold. 
If $\omega_0\in \Omega$ satisfies $d''(\omega_0)<0$, 
then the bound state $\Trans (\omega_0 t)\phi_{\omega_0}$ is unstable. 
\end{corollary}

\begin{corollary}\label{cor4}
Assume $({\rm B1})$ and that for each $\omega\in \Omega$, 
$({\rm B2b})$ and $({\rm B3})$ hold. 
If $\omega_0\in \Omega$ satisfies $d''(\omega_0)>0$, 
then the bound state $\Trans (\omega_0 t)\phi_{\omega_0}$ is unstable. 
\end{corollary}

\begin{remark}\label{rem3}
Under the assumptions (B1), (B2a) and (B3), 
it is proved that if $\omega_0\in \Omega$ satisfies $d''(\omega_0)>0$, 
then the bound state $\Trans (\omega_0 t)\phi_{\omega_0}$ is stable 
(see Section 3 of \cite{GSS1}). 
\end{remark}

\begin{remark}\label{rem4}
When $S_{\omega}''(\phi_{\omega})$ has two or more negative eigenvalues, 
linear instability of $\Trans (\omega t)\phi_{\omega}$ is studied 
by many authors (see, e.g., \cite{ES,gri,GSS2,jon,KKSW}). 
However, it is a non-trivial problem 
whether linear instability implies (nonlinear) instability. 
For a recent development in this direction, see \cite{GO}. 
Corollary \ref{cor4} gives a sufficient condition for instability of bound states 
without using the argument through linear instability 
(see also Subsection \ref{ss:delta}). 
This was the main assertion in \cite{mae2}. 
\end{remark}

\section{Preliminaries}\label{sect:pre}

In this section we assume (A1). 
Recall that $\Trans$ is $2\pi$-periodic. We often use the relations 
$R\Trans (s)=\tilde \Trans (s)R$, $RJ=\tilde JR$, 
$I\Trans (s)=\tilde \Trans (s)I$, $IJ=\tilde JI$, 
which follow from the definitions of $R$, $I$, $\tilde \Trans(s)$ and $\tilde J$ 
in Section \ref{sect:form}. 

\begin{lemma}\label{lem1}
There exist $\varepsilon>0$ and a $C^2$ map 
$\theta:\mathcal{N}_{\varepsilon}(\phi_{\omega})\to \R/2\pi \Z$ such that 
for all $u\in \mathcal{N}_{\varepsilon}(\phi_{\omega})$ and all $s \in \R/2\pi \Z$, 
\begin{align}
&\|\Trans (\theta (u)) u-\phi_{\omega}\|_X\le \|\Trans (s)u-\phi_{\omega}\|_X, \nonumber \\ 
&(\Trans (\theta (u)) u, J\phi_{\omega})_X=0, \quad 
\theta (\Trans (s)u)=\theta (u)-s, \nonumber \\ 
&\theta'(u)=\frac{R\Trans (-\theta(u))J\phi_{\omega}}
{(J^2\phi_{\omega},\Trans (\theta (u))u)_X}\in I(X). \label{eq:4.1}
\end{align}
\end{lemma}

\begin{proof}
See Lemma 3.2 of \cite{GSS1}. We remark that $\theta'(u)\in I(X)$ follows from 
the assumption $R\phi_{\omega}\in I(X)$ in (A1). 
\end{proof}

For $u\in \mathcal{N}_{\varepsilon}(\phi_{\omega})$, 
we define $M(u)=\Trans (\theta (u))u$, and 
\begin{equation}\label{eq:4.2}
A(u)=(M(u),J^{-1}\psi)_H, \quad \Lambda (u)=(M(u),\psi)_H. 
\end{equation}
Then we have 
$$\dual{A'(u)}{v}=(\Trans (\theta(u))v,J^{-1}\psi)_H-\Lambda (u)\dual{\theta'(u)}{v}$$
for $v\in X$. By Lemma \ref{lem1}, we see that $A'(u)\in I(X)$ and 
\begin{equation}\label{eq:4.3}
JI^{-1}A'(u)=\Trans (-\theta (u))\psi-\Lambda (u)JI^{-1}\theta'(u)
\end{equation}
for $u\in \mathcal{N}_{\varepsilon}(\phi_{\omega})$. 
Moreover, since $A$ is invariant under $\Trans$, we have 
\begin{equation}\label{eq:4.4}
0=\frac{d}{ds}A(\Trans(s)u)|_{s=0}
=\dual{A'(u)}{Ju}=-\dual{Q'(u)}{JI^{-1}A'(u)}.
\end{equation}
We define $P$ by 
$$P(u)=\dual{E'(u)}{JI^{-1}A'(u)}$$
for $u\in \mathcal{N}_{\varepsilon}(\phi_{\omega})$. 
By \eqref{eq:4.4}, we have 
$P(u)=\dual{S_{\omega}'(u)}{JI^{-1}A'(u)}$. 
Moreover, by \eqref{eq:4.1}, \eqref{eq:4.3} 
and by \eqref{eq:2.2}, \eqref{eq:2.4}, \eqref{eq:2.5}, we see that 
\begin{equation}\label{eq:4.5}
P(u)=\dual{S_{\omega}'(M(u))}{\psi}-\Lambda (u)
\frac{\dual{S_{\omega}'(M(u))}{JI^{-1}RJ\phi_{\omega}}}{(M(u),J^2\phi_{\omega})_{X}}. 
\end{equation}

\begin{lemma}\label{lem2}
Let $\mathcal{I}$ be an interval of $\R$. 
Let $u\in C(\mathcal{I},X)\cap C^1(\mathcal{I},X^*)$ be a solution of \eqref{eq:2.3}, 
and assume that $u(t)\in \mathcal{N}_{\varepsilon}(\phi_{\omega})$ 
for all $t\in \mathcal{I}$. Then 
$$\frac{d}{dt}A(u(t))=-P(u(t))$$
for all $t\in \mathcal{I}$. 
\end{lemma}

\begin{proof}
By Lemma 4.6 of \cite{GSS1}, we see that $t\mapsto A(u(t))$ 
is a $C^1$ function on $\mathcal{I}$, and 
$$\frac{d}{dt}A(u(t))
=\dual{\partial_t u(t)}{I^{-1}A'(u(t))}$$
for all $t\in \mathcal{I}$. 
Since $u(t)$ is a solution of \eqref{eq:2.3}, we have 
\begin{align*}
&\dual{\partial_t u(t)}{I^{-1}A'(u(t))}=\dual{\tilde J E'(u(t))}{I^{-1}A'(u(t))} \\
&=-\dual{E'(u(t))}{JI^{-1}A'(u(t))}=-P(u(t))
\end{align*}
for $t\in \mathcal{I}$. This completes the proof. 
\end{proof}

\section{Proof of Theorem \ref{thm1}} \label{sect:proofthm1}

In this section we make the same assumptions as in Theorem \ref{thm1}. We define 
\begin{equation}\label{eq:5.1}
W=\{w\in X: (\phi_{\omega},w)_H=(J\phi_{\omega},w)_H=(\psi,w)_H=0\}.
\end{equation}

\begin{lemma}\label{lem3}
There exists $\varepsilon_0>0$ such that 
$$E(u)\ge E(\phi_{\omega})+\Lambda (u) P(u)$$
for all $u\in \mathcal{N}_{\varepsilon_0}(\phi_{\omega})$ 
satisfying $Q(u)=Q(\phi_{\omega})$. 
\end{lemma}

\begin{proof}
We put $v=M(u)-\phi_{\omega}$, and decompose $v$ as 
$$v=a\phi_{\omega}+bJ\phi_{\omega}+c\psi+w,$$
where $a$, $b$, $c\in \R$ and $w\in W$. 
Note that $\|v\|_{X}<\varepsilon_0$. Since 
$$Q(\phi_{\omega})=Q(u)=Q(M(u))=Q(\phi_{\omega})+(\phi_{\omega},v)_{H}+Q(v),$$ 
we have $(\phi_{\omega},v)_{H}=a\|\phi_{\omega}\|_{H}^2=-Q(v)$. 
In particular, $a=O(\|v\|_X^2)$. 
Moreover, by \eqref{eq:2.1} and Lemma \ref{lem1}, 
we have $(\phi_{\omega},J\phi_{\omega})_X=(M(u),J\phi_{\omega})_X=0$. Thus, 
$$0=(v,J\phi_{\omega})_X=b\|J\phi_{\omega}\|_X^2+(c\psi+w,J\phi_{\omega})_X,$$
$\|bJ\phi_{\omega}\|_X\le \|c\psi\|_X+\|w\|_X$, and 
\begin{equation}\label{eq:5.2}
2|c|\|\psi\|_X+2\|w\|_X\ge \|v\|_X-O(\|v\|_{X}^2).
\end{equation}
Since $S_{\omega}'(\phi_{\omega})=0$ and $Q(u)=Q(\phi_{\omega})$, 
by the Taylor expansion, we have 
\begin{equation}\label{eq:5.3}
E(u)-E(\phi_{\omega})=S_{\omega}(M(u))-S_{\omega}(\phi_{\omega}) 
=\frac{1}{2}\dual{S_{\omega}''(\phi_{\omega})v}{v}+o(\|v\|_{X}^2). 
\end{equation}
Here, since $a=O(\|v\|_X^2)$ and $S_{\omega}''(\phi_{\omega})(J\phi_{\omega})=0$, 
we have 
\begin{align}
&\dual{S_{\omega}''(\phi_{\omega})v}{v}
=\dual{S_{\omega}''(\phi_{\omega})(c\psi+w)}{c\psi+w}
+o(\|v\|_{X}^2) \nonumber \\
&=c^2\dual{S_{\omega}''(\phi_{\omega})\psi}{\psi}
+2c \dual{S_{\omega}''(\phi_{\omega})\psi}{w}
+\dual{S_{\omega}''(\phi_{\omega})w}{w}+o(\|v\|_X^2). \label{eq:5.4}
\end{align}
On the other hand, we have $c=(v,\psi)_H=\Lambda (u)=O(\|v\|_X)$ and 
\begin{align*}
&S_{\omega}'(\phi_{\omega}+v)
=S_{\omega}'(\phi_{\omega})+S_{\omega}''(\phi_{\omega})v+o(\|v\|_X)
=S_{\omega}''(\phi_{\omega})v+o(\|v\|_X), \\
&(M(u),J^2\phi_{\omega})_X=(\phi_{\omega}+v,J^2\phi_{\omega})_X
=-\|J\phi_{\omega}\|_X^2+O(\|v\|_X).
\end{align*}
Thus, by \eqref{eq:4.5}, we have
\begin{align}
\Lambda (u) P(u)
&=c\dual{S_{\omega}''(\phi_{\omega})v}{\psi}+o(\|v\|_X^2) \nonumber \\
&=c^2\dual{S_{\omega}''(\phi_{\omega})\psi}{\psi}
+c\dual{S_{\omega}''(\phi_{\omega})\psi}{w}+o(\|v\|_X^2). \label{eq:5.5}
\end{align}
By \eqref{eq:5.3}, \eqref{eq:5.4} and \eqref{eq:5.5}, we have 
\begin{align}
&E(u)-E(\phi_{\omega})-\Lambda (u) P(u) \nonumber \\
&=-\frac{c^2}{2}\dual{S_{\omega}''(\phi_{\omega})\psi}{\psi}
+\frac{1}{2}\dual{S_{\omega}''(\phi_{\omega})w}{w}
+o(\|v\|_X^2). \label{eq:5.6}
\end{align}
Here, by the assumptions (A2a) and (A3), 
there exists a positive constant $k>0$ such that 
$$-\frac{c^2}{2}\dual{S_{\omega}''(\phi_{\omega})\psi}{\psi}
+\frac{1}{2}\dual{S_{\omega}''(\phi_{\omega})w}{w}
\ge k(c^2+\|w\|_X^2).$$
Moreover, since $\|v\|_X=\|M(u)-\phi_{\omega}\|_X<\varepsilon_0$, 
it follows from \eqref{eq:5.2} that 
the right hand side of \eqref{eq:5.6} is non-negative, 
if $\varepsilon_0$ is sufficiently small. 
This completes the proof. 
\end{proof}

\begin{lemma}\label{lem4}
There exist $\lambda_1>0$ and a smooth mapping 
$\lambda\mapsto \varphi_{\lambda}$ from $(-\lambda_1,\lambda_1)$ to $X$ 
such that $\varphi_{0}=\phi_{\omega}$ and 
$$E(\varphi_{\lambda})<E(\phi_{\omega}), \quad 
Q(\varphi_{\lambda})=Q(\phi_{\omega}), \quad \lambda P(\varphi_{\lambda})<0 
\quad \mbox{for} \quad 0<|\lambda|<\lambda_1.$$
\end{lemma}

\begin{proof}
For $\lambda$ close to $0$, we define 
$$\varphi_{\lambda}=\phi_{\omega}+\lambda \psi+\sigma(\lambda)\phi_{\omega}, \quad 
\sigma (\lambda)=\left(1-\frac{Q(\psi)}{Q(\phi_{\omega})}\lambda^2\right)^{1/2}-1.$$
Then, we have $Q(\varphi_{\lambda})=Q(\phi_{\omega})$, 
$\sigma (\lambda)=O(\lambda^2)$, and 
\begin{align*}
&S_{\omega}(\varphi_{\lambda})=S_{\omega}(\phi_{\omega})
+\frac{\lambda^2}{2}\dual{S_{\omega}''(\phi_{\omega})\psi}{\psi}+o(\lambda^2), \\
&S_{\omega}'(\varphi_{\lambda})
=\lambda S_{\omega}''(\phi_{\omega})\psi+o(\lambda), \quad 
P(\varphi_{\lambda})=\lambda \dual{S_{\omega}''(\phi_{\omega})\psi}{\psi}+o(\lambda)
\end{align*}
as $\lambda\to 0$. This completes the proof. 
\end{proof}

\begin{proof}[Proof of Theorem \ref{thm1}]
Suppose that $\Trans (\omega t)\phi_{\omega}$ is stable. 
For $\lambda$ close to $0$, let $\varphi_{\lambda}\in X$ 
be the vector given in Lemma \ref{lem4}, 
and let $u_{\lambda}(t)$ be the solution of \eqref{eq:2.3} 
with $u_{\lambda}(0)=\varphi_{\lambda}$. 
Then, there exists $\lambda_0>0$ such that if $|\lambda|<\lambda_0$, 
then $u_{\lambda}(t)\in \mathcal{N}_{\varepsilon_0}(\phi_{\omega})$ for all $t\ge 0$, 
where $\varepsilon_0$ is the positive constant given in Lemma \ref{lem3}. 
Moreover, by the definition \eqref{eq:4.2} of $A$ and $\Lambda$, 
there exist positive constants $C_1$ and $C_2$ such that 
$|A(v)|\le C_1$ and $|\Lambda(v)|\le C_2$ 
for all $v\in \mathcal{N}_{\varepsilon_0}(\phi_{\omega})$. 
Let $\lambda \in (0,\lambda_0)$ and 
put $\delta_{\lambda}=E(\phi_{\omega})-E(\varphi_{\lambda})>0$. 
Since $P(\varphi_{\lambda})<0$ and $t\mapsto P(u_{\lambda}(t))$ is continuous, 
by Lemma \ref{lem3} and conservation of $E$ and $Q$, 
we see that $P(u_{\lambda}(t))<0$ for all $t\ge 0$ and that 
$$\delta_{\lambda}=E(\phi_{\omega})-E(u_{\lambda}(t))
\le -\Lambda (u_{\lambda}(t))P(u_{\lambda}(t))
\le -C_2 P(u_{\lambda}(t))$$
for all $t\ge 0$. Moreover, by Lemma \ref{lem2}, we have 
$$\frac{d}{dt}A(u_{\lambda}(t))=-P(u_{\lambda}(t))\ge \delta_{\lambda}/C_2$$
for all $t\ge 0$, which implies that $A(u_{\lambda}(t))\to \infty$ as $t\to \infty$. 
This contradicts the fact that $|A(u_{\lambda}(t))|\le C_1$ for all $t\ge 0$. 
Hence, $\Trans (\omega t)\phi_{\omega}$ is unstable. 
\end{proof}

\section{Proof of Theorem \ref{thm2}}\label{sect:proofthm2}

In this section we make the same assumptions as in Theorem \ref{thm2}. 
We modify the argument in the previous section to prove Theorem \ref{thm2}. 
We put 
\begin{equation}\label{eq:6.1}
\nu:=3\mu-\dual{S_{\omega}'''(\phi_{\omega})(\psi,\psi)}{\psi}.
\end{equation}
By the assumption \eqref{eq:3.1}, $\nu\ne 0$. 

\begin{lemma}\label{lem5}
There exist positive constants $\varepsilon_0$ and $k^*$ such that 
$$E(u)\ge E(\phi_{\omega})+\frac{\nu}{|\nu|}k^* P(u)$$
for all $u\in \mathcal{N}_{\varepsilon_0}(\phi_{\omega})$ 
satisfying $Q(u)=Q(\phi_{\omega})$. 
\end{lemma}

\begin{proof}
We put $v=M(u)-\phi_{\omega}$, and decompose $v$ as 
$$v=a\phi_{\omega}+bJ\phi_{\omega}+c\psi+w,$$
where $a$, $b$, $c\in \R$, $w\in W$, 
and $W$ is the set defined by \eqref{eq:5.1}. 
Then we have $(\phi_{\omega},v)_{H}=a\|\phi_{\omega}\|_{H}^2=-Q(v)$. 
Moreover, by \eqref{eq:2.1}, (A2b) and Lemma \ref{lem1}, we have 
$(\phi_{\omega},J\phi_{\omega})_X=(\psi,J\phi_{\omega})_X=(M(u),J\phi_{\omega})_X=0$. 
Thus, $0=(v,J\phi_{\omega})_X=b\|J\phi_{\omega}\|_X^2+(w,J\phi_{\omega})_X$, and 
\begin{equation}\label{eq:6.2}
\|bJ\phi_{\omega}\|_X\le \|w\|_X, \quad 
|c|\|\psi\|_X+2\|w\|_X\ge \|v\|_X-O(\|v\|_{X}^2).
\end{equation}
We also have \eqref{eq:5.3}. Here, by Remark \ref{rem2}, we have 
\begin{align}
\dual{S_{\omega}''(\phi_{\omega})v}{v}
&=c^2\dual{S_{\omega}''(\phi_{\omega})\psi}{\psi}
+2c \dual{S_{\omega}''(\phi_{\omega})\psi}{w}+\dual{S_{\omega}''(\phi_{\omega})w}{w}
+o(\|v\|_{X}^2) \nonumber \\
&=\dual{S_{\omega}''(\phi_{\omega})w}{w}+o(\|v\|_{X}^2). \label{eq:6.3}
\end{align}
By \eqref{eq:5.3}, \eqref{eq:6.3} and (A3), we have 
\begin{equation}\label{eq:6.4}
E(u)-E(\phi_{\omega})
=\frac{1}{2}\dual{S_{\omega}''(\phi_{\omega})w}{w}+o(\|v\|_{X}^2)
\ge \frac{k_0}{2}\|w\|_X^2-o(\|v\|_{X}^2).
\end{equation}
On the other hand, we have $c=(v,\psi)_H=\Lambda (u)=O(\|v\|_X)$ and 
\begin{align*}
&S_{\omega}'(\phi_{\omega}+v)
=S_{\omega}''(\phi_{\omega})v+\frac{1}{2}S_{\omega}'''(\phi_{\omega})(v,v)
+o(\|v\|_{X}^2), \\
&(M(u),J^2\phi_{\omega})_X
=(\phi_{\omega}+v,J^2\phi_{\omega})_X=-\|J\phi_{\omega}\|_X^2+O(\|v\|_X). 
\end{align*}
Thus, by \eqref{eq:4.5}, we have 
\begin{align*}
P(u)=\dual{S_{\omega}''(\phi_{\omega})\psi}{v}
&+\frac{1}{2}\dual{S_{\omega}'''(\phi_{\omega})(v,v)}{\psi} \\
&+\frac{c}{\|J\phi_{\omega}\|_{X}^2}
\dual{S_{\omega}''(\phi_{\omega})v}{JI^{-1}RJ\phi_{\omega}}+o(\|v\|_{X}^2).
\end{align*}
Here, by \eqref{eq:3.1} and \eqref{eq:6.2}, we have 
\begin{align*}
\dual{S_{\omega}''(\phi_{\omega})\psi}{v}
&=\mu (\phi_{\omega},v)_{H}=-\mu Q(v)=-\frac{\mu}{2}\|v\|_H^2 \\
&=-\frac{\mu}{2}\left\{a^2\|\phi_{\omega}\|_{H}^2
+b^2\|J\phi_{\omega}\|_{H}^2+c^2\|\psi\|_{H}^2+\|w\|_{H}^2\right\} \\
&=-\frac{c^2\mu}{2}+O(\|w\|_X^2)+o(\|v\|_X^2),
\end{align*} 
\begin{align*}
\dual{S_{\omega}'''(\phi_{\omega})(v,v)}{\psi}
=c^2 & \dual{S_{\omega}'''(\phi_{\omega})(\psi,\psi)}{\psi} 
+2c\dual{S_{\omega}'''(\phi_{\omega})(\psi,bJ\phi_{\omega}+w)}{\psi} \\
&+O(\|w\|_X^2)+o(\|v\|_X^2), 
\end{align*} 
\begin{align*}
&c\dual{S_{\omega}''(\phi_{\omega})v}{JI^{-1}RJ\phi_{\omega}}
=c\dual{S_{\omega}''(\phi_{\omega})(c\psi+w)}{JI^{-1}RJ\phi_{\omega}}+o(\|v\|_X^2) \\
&=-c^2\mu \|J\phi_{\omega}\|_X^2
+c\dual{S_{\omega}''(\phi_{\omega})w}{JI^{-1}RJ\phi_{\omega}}+o(\|v\|_X^2).
\end{align*}
Therefore, there exists a constant $k>0$ such that 
$$\left|P(u)+\frac{\nu}{2}c^2\right| 
\le k \left(|c|\|w\|_{X}+\|w\|_{X}^2\right)+o(\|v\|_{X}^2),$$
where $\nu$ is the constant defined by \eqref{eq:6.1}. 
Thus, there exists a constant $k_1>0$ such that 
\begin{equation}\label{eq:6.5}
-\frac{\nu}{|\nu|}P(u)\ge \frac{|\nu|}{4}c^2-k_1 \|w\|_{X}^2-o(\|v\|_{X}^2).
\end{equation}
By \eqref{eq:6.4} and \eqref{eq:6.5}, we have 
\begin{equation}\label{eq:6.6}
E(u)-E(\phi_{\omega})-\frac{\nu}{|\nu|}k^*P(u)
\ge k_2 c^2+k_3\|w\|_{X}^2-o(\|v\|_{X}^2),
\end{equation}
where $k^*={k_0}/4k_1$, $k_2=k^*|\nu|/4$ and $k_3=k_0/4$. 
Finally, since $\|v\|_X=\|M(u)-\phi_{\omega}\|_X<\varepsilon_0$, 
it follows from \eqref{eq:6.2} that 
the right hand side of \eqref{eq:6.6} is non-negative, 
if $\varepsilon_0$ is sufficiently small. This completes the proof. 
\end{proof}

\begin{lemma}\label{lem6}
There exist $\lambda_1>0$ and a smooth mapping $\lambda\mapsto \varphi_{\lambda}$ 
from $(-\lambda_1,\lambda_1)$ to $X$ such that $\varphi_{0}=\phi_{\omega}$ and 
$$E(\varphi_{\lambda})<E(\phi_{\omega}), \quad Q(\varphi_{\lambda})=Q(\phi_{\omega}) 
\quad \mbox{for} \quad 0<\frac{\nu}{|\nu|}\lambda<\lambda_1.$$
\end{lemma}

\begin{proof}
For $\lambda$ close to $0$, we define 
$$\varphi_{\lambda}=\phi_{\omega}+\lambda \psi+\sigma(\lambda)\phi_{\omega}, \quad 
\sigma (\lambda)=\left(1-\frac{Q(\psi)}{Q(\phi_{\omega})}\lambda^2\right)^{1/2}-1.$$
Then, we have $Q(\varphi_{\lambda})=Q(\phi_{\omega})$ and 
\begin{align*}
&\sigma (\lambda)=-\frac{1}{2\|\phi_{\omega}\|_H^2}\lambda^2+O(\lambda^4), \\
&S_{\omega}(\varphi_{\lambda})=S_{\omega}(\phi_{\omega})+\frac{\lambda^2}{2}
\dual{S_{\omega}''(\phi_{\omega})\psi}{\psi}+\lambda \sigma (\lambda)
\dual{S_{\omega}''(\phi_{\omega})\psi}{\phi_{\omega}} \\
&\hspace{20mm}
+\frac{\lambda^3}{6}\dual{S_{\omega}'''(\phi_{\omega})(\psi,\psi)}{\psi}+o(\lambda^3).
\end{align*}
Here, by \eqref{eq:3.1} we have 
$\dual{S_{\omega}''(\phi_{\omega})\psi}{\psi}=\mu (\phi_{\omega},\psi)_H=0$ and 
$\dual{S_{\omega}''(\phi_{\omega})\psi}{\phi_{\omega}} =\mu \|\phi_{\omega}\|_H^2$. Thus, 
$$S_{\omega}(\varphi_{\lambda})
=S_{\omega}(\phi_{\omega})-\frac{\nu}{6}\lambda^3+o(\lambda^3).$$
This completes the proof. 
\end{proof}

By Lemmas \ref{lem5} and \ref{lem6}, we can prove Theorem \ref{thm2} 
in the same way as in the proof of Theorem \ref{thm1}. We omit the detail. 

\section{Proofs of Corollaries}\label{sect:proofcor}

In this section we prove Corollaries \ref{cor2}, \ref{cor3} and \ref{cor4}. 
We first give a sufficient condition for (A3). 

\begin{lemma}\label{lem7}
Assume $({\rm B2a})$ and $({\rm B3})$. Assume further that there exist 
$\psi\in X$ and constants $\lambda\le 0$ and $\mu\in \R$ such that 
$\|\psi\|_H=1$, $(\phi_{\omega},\psi)_H=(J\phi_{\omega},\psi)_H=0$ 
and $S_{\omega}''(\phi_{\omega})\psi=\lambda I\psi+\mu Q'(\phi_{\omega})$. 
Then $({\rm A3})$ holds. 
\end{lemma}

\begin{proof}
First we claim that $\dual{S_{\omega}''(\phi_{\omega})w}{w}>0$ 
for all $w\in X$ satisfying $w\ne 0$ and 
$(\phi_{\omega},w)_H=(J\phi_{\omega},w)_H=(\psi,w)_H=0$. 
We prove this by contradiction. 
Suppose that there exists $w_0\in X$ such that 
$\dual{S_{\omega}''(\phi_{\omega})w_0}{w_0}\le 0$, $w_0\ne 0$ and 
$(\phi_{\omega},w_0)_H=(J\phi_{\omega},w_0)_H=(\psi,w_0)_H=0$. 
Then there exists $(\alpha,\beta)\in \R^2$ such that 
$(\alpha,\beta)\ne (0,0)$ and $(\alpha \psi+\beta w_0,\chi_{\omega})_H=0$. 
We put $p=\alpha \psi+\beta w_0$. 
Then $p\in X$ satisfies $(\chi_{\omega},p)_H=(J\phi_{\omega},p)_H=0$ 
and $p\ne 0$. Thus, by (B2a), we have $\dual{S_{\omega}''(\phi_{\omega})p}{p}>0$. 
On the other hand, we have 
\begin{align*}
&\dual{S_{\omega}''(\phi_{\omega})\psi}{w_0} 
=\lambda (\psi,w_0)_H+\mu (\phi_{\omega},w_0)_H=0, \\
&\dual{S_{\omega}''(\phi_{\omega})p}{p} 
=\alpha^2\lambda \|\psi\|_H^2
+2\alpha \beta \dual{S_{\omega}''(\phi_{\omega})\psi}{w_0}
+\beta^2\dual{S_{\omega}''(\phi_{\omega})w_0}{w_0}\le 0.
\end{align*}
This contradiction proves our first claim. 
Next we prove (A3) by contradiction. Suppose that (A3) does not hold. 
Then there exists a sequence $(w_n)$ in $X$ such that 
$\dual{S_{\omega}''(\phi_{\omega})w_n}{w_n}\to 0$, 
$\|w_n\|_X=1$ and $(\phi_{\omega},w_n)_H=(J\phi_{\omega},w_n)_H=(\psi,w_n)_H=0$. 
There exist a subsequence $(w_{n'})$ of $(w_n)$ and $w\in X$ 
such that $w_{n'}\rightharpoonup w$ weakly in $X$. 
By (B3), we see that $w\ne 0$, 
$(\phi_{\omega},w)_H=(J\phi_{\omega},w)_H=(\psi,w)_H=0$ and 
$$\dual{S_{\omega}''(\phi_{\omega})w}{w}\le \liminf_{n'\to \infty}
\dual{S_{\omega}''(\phi_{\omega})w_{n'}}{w_{n'}}=0.$$
However, this contradicts the first claim. This completes the proof. 
\end{proof}

\begin{proof}[Proof of Corollary \ref{cor2}]
We verify that $\phi_{\omega_0}$ satisfies the assumptions 
(A1), (A2b) and (A3) of Theorem \ref{thm2}. 
First (A1) follows from (B1). 
Next, by (B1), $E'(\phi_{\omega})=\omega Q'(\phi_{\omega})$ for all $\omega\in \Omega$. 
Differentiating this with respect to $\omega$, we have 
\begin{equation}\label{eq:7.1}
S_{\omega}''(\phi_{\omega})\phi_{\omega}'=Q'(\phi_{\omega}), \quad 
S_{\omega}'''(\phi_{\omega})(\phi_{\omega}',\phi_{\omega}')
+S_{\omega}''(\phi_{\omega})\phi_{\omega}''=2Q''(\phi_{\omega})\phi_{\omega}',
\end{equation}
where $\phi_{\omega}'=d\phi_{\omega}/d\omega$ 
and $\phi_{\omega}''=d^2\phi_{\omega}/d\omega^2$. 
While, differentiating $d(\omega)=E(\phi_{\omega})-\omega Q(\phi_{\omega})$, 
we have 
\begin{align}
d'(\omega)&=\dual{E'(\phi_{\omega})}{\phi_{\omega}'}
-\omega \dual{Q'(\phi_{\omega})}{\phi_{\omega}'}
-Q(\phi_{\omega})=-Q(\phi_{\omega}), \nonumber \\
d''(\omega)&=-\dual{Q'(\phi_{\omega})}{\phi_{\omega}'}
=-(\phi_{\omega},\phi_{\omega}')_H
=-\dual{S_{\omega}''(\phi_{\omega})\phi_{\omega}'}{\phi_{\omega}'}. 
\label{eq:7.2}
\end{align}
Moreover, by \eqref{eq:7.1} and \eqref{eq:7.2}, we have 
\begin{align}
d'''(\omega)
&=-\dual{Q''(\phi_{\omega})\phi_{\omega}'}{\phi_{\omega}'}
-\dual{Q'(\phi_{\omega})}{\phi_{\omega}''} \nonumber \\
&=\dual{S_{\omega}'''(\phi_{\omega})(\phi_{\omega}',\phi_{\omega}')}{\phi_{\omega}'}
-3\dual{Q''(\phi_{\omega})\phi_{\omega}'}{\phi_{\omega}'} \nonumber \\
&=\dual{S_{\omega}'''(\phi_{\omega})(\phi_{\omega}',\phi_{\omega}')}
{\phi_{\omega}'}-3\|\phi_{\omega}'\|_H^2. \label{eq:7.3}
\end{align}
Here we take 
$$\mu=\frac{1}{\|\phi_{\omega_0}'\|_H}, \quad \psi=\mu \phi_{\omega_0}'.$$
Then, $\|\psi\|_H=1$ and 
$S_{\omega_0}''(\phi_{\omega_0})\psi=\mu Q'(\phi_{\omega_0})$. 
By (B1), we have $(J\phi_{\omega_0},\psi)_H=(J\phi_{\omega_0},\psi)_X=0$. 
Moreover, since $d''(\omega_0)=0$ and $d'''(\omega_0)\ne 0$, 
by \eqref{eq:7.2} and \eqref{eq:7.3}, we have $(\phi_{\omega_0},\psi)_H=0$ and 
$$\dual{S_{\omega_0}'''(\phi_{\omega_0})(\psi,\psi)}{\psi}
=\mu^3\dual{S_{\omega_0}'''(\phi_{\omega_0})
(\phi_{\omega_0}',\phi_{\omega_0}')}{\phi_{\omega_0}'}\ne 3\mu.$$ 
Thus, (A2b) is verified. 
Finally, (A3) follows from (A2b) and Lemma \ref{lem7}. 
\end{proof}

The following lemma is used in the proof of Corollary \ref{cor3}. 

\begin{lemma}\label{lem8}
Assume $({\rm B1})$ and that for each $\omega\in \Omega$, 
$({\rm B2a})$ and $({\rm B3})$ hold. 
If $\omega_0\in \Omega$ satisfies $d''(\omega_0)<0$, 
then there exist $\psi\in X$ and constants $\lambda<0$ and $\mu\in \R$ 
such that $\|\psi\|_H=1$, $(\phi_{\omega_0},\psi)_H=(J\phi_{\omega_0},\psi)_H=0$ 
and $S_{\omega_0}''(\phi_{\omega_0})\psi=\lambda I\psi+\mu Q'(\phi_{\omega_0})$. 
\end{lemma}

\begin{proof}
We define 
\begin{equation}\label{eq:7.4}
\lambda=\inf\{\dual{S_{\omega_0}''(\phi_{\omega_0})w}{w}:
w\in X,~ \|w\|_H=1,~ (\phi_{\omega_0},w)_H=0\}.
\end{equation}
By Theorem 4.1 of \cite{GSS1} and by \eqref{eq:3.3} in (B3), 
we see that $-\infty<\lambda<0$. 
Moreover, by the standard variational argument with (B3) 
(see, e.g., Chapter 11 of \cite{LL}), 
we see that \eqref{eq:7.4} is attained at some $\psi$, 
that is, there exists $\psi\in X$ such that 
$\dual{S_{\omega_0}''(\phi_{\omega_0})\psi}{\psi}=\lambda$, 
$\|\psi\|_H=1$ and $(\phi_{\omega_0},\psi)_H=0$. 
Then there exists a Lagrange multiplier $\mu\in \R$ such that 
$S_{\omega_0}''(\phi_{\omega_0})\psi=\lambda I\psi+\mu Q'(\phi_{\omega_0})$. 
Finally, by this equation, we have 
$$\lambda (\psi,J\phi_{\omega_0})_H
=\dual{S_{\omega_0}''(\phi_{\omega_0})(J\phi_{\omega_0})}{\psi}
-\mu (\phi_{\omega_0},J\phi_{\omega_0})_H=0.$$
Since $\lambda\ne 0$, we have $(J\phi_{\omega_0},\psi)_H=0$. 
This completes the proof. 
\end{proof}

\begin{proof}[Proof of Corollary \ref{cor3}]
We verify that $\phi_{\omega_0}$ satisfies the assumptions 
(A1), (A2a) and (A3) of Theorem \ref{thm1}. 
(A1) follows from (B1), and (A2a) follows from Lemma \ref{lem8}. 
Finally, (A3) follows from Lemmas \ref{lem7} and \ref{lem8}. 
\end{proof}

The following lemma is based on Theorem 2 of \cite{mae2} 
(see also Theorem 3.3 of \cite{GSS1}), 
and is used in the proof of Corollary \ref{cor4}. 

\begin{lemma}\label{lem9}
Assume $({\rm B1})$ and that for each $\omega\in \Omega$, 
$({\rm B2b})$ and $({\rm B3})$ hold. 
If $\omega_0\in \Omega$ satisfies $d''(\omega_0)>0$, 
then there exists a constant $k_0>0$ such that
$$\dual{S_{\omega_0}''(\phi_{\omega_0})w}{w}\ge k_0 \|w\|_X^2$$
for all $w\in X$ satisfying 
$(\phi_{\omega_0},w)_H=(\chi_{1,\omega_0},w)_H=(J\phi_{\omega_0},w)_H=0$. 
\end{lemma}

\begin{proof}
As in the proof of Lemma \ref{lem7}, it suffices to prove that 
$\dual{S_{\omega_0}''(\phi_{\omega_0})w}{w}>0$ 
for all $w\in X$ satisfying $w\ne 0$ and 
$(\phi_{\omega_0},w)_H=(\chi_{1,\omega_0},w)_H=(J\phi_{\omega_0},w)_H=0$. 
We define 
$$P_{\omega}=\{p\in X:
(\chi_{0,\omega},p)_H=(\chi_{1,\omega},p)_H=(J\phi_{\omega},p)_H=0\}.$$
Let $w\in X$ satisfy $w\ne 0$ and 
$(\phi_{\omega_0},w)_H=(\chi_{1,\omega_0},w)_H=(J\phi_{\omega_0},w)_H=0$. 
We decompose $w$ and $\phi_{\omega_0}'$ as 
\begin{align*}
w&=a_0\chi_{0,\omega_0}+a_1\chi_{1,\omega_0}+a_2J\phi_{\omega_0}+p, \\
\phi_{\omega_0}'&=b_0\chi_{0,\omega_0}+b_1\chi_{1,\omega_0}+b_2J\phi_{\omega_0}+q,
\end{align*}
where $a_j,b_j\in \R$ and $p,q\in P_{\omega_0}$. 
Since $(\chi_{1,\omega_0},w)_H=(J\phi_{\omega_0},w)_H=0$, we have $a_1=a_2=0$. 
Moreover, by the first equation of \eqref{eq:7.1}, 
$$\lambda_{1,\omega_0}b_1
=(\lambda_{1,\omega_0}\chi_{1,\omega_0},\phi_{\omega_0}')_H
=\dual{S_{\omega_0}''(\phi_{\omega_0})\chi_{1,\omega_0}}{\phi_{\omega_0}'}
=(\chi_{1,\omega_0},\phi_{\omega_0})_H=0.$$
Thus, $b_1=0$. By \eqref{eq:7.2}, we have 
$$0>-d''(\omega_0)
=\dual{S_{\omega_0}''(\phi_{\omega_0})\phi_{\omega_0}'}{\phi_{\omega_0}'}
=b_0^2\lambda_{0,\omega_0}+\dual{S_{\omega_0}''(\phi_{\omega_0})q}{q}.$$
In particular, $b_0\ne 0$. 
On the other hand, by the first equation of \eqref{eq:7.1}, 
$$0=(\phi_{\omega_0},w)_H
=\dual{S_{\omega_0}''(\phi_{\omega_0})\phi_{\omega_0}'}{w}
=a_0b_0\lambda_{0,\omega_0}+\dual{S_{\omega_0}''(\phi_{\omega_0})q}{p}.$$
In particular, $p\ne 0$. By the Cauchy-Schwarz inequality, we have
\begin{align*}
&b_0^2|\lambda_{0,\omega_0}|\dual{S_{\omega_0}''(\phi_{\omega_0})w}{w} 
=b_0^2|\lambda_{0,\omega_0}|\{a_0^2\lambda_{0,\omega_0}
+\dual{S_{\omega_0}''(\phi_{\omega_0})p}{p}\} \\
&>-a_0^2b_0^2\lambda_{0,\omega_0}^2+\dual{S_{\omega_0}''(\phi_{\omega_0})p}{p}
\dual{S_{\omega_0}''(\phi_{\omega_0})q}{q} \\
&\ge -a_0^2b_0^2\lambda_{0,\omega_0}^2
+\dual{S_{\omega_0}''(\phi_{\omega_0})q}{p}^2=0.
\end{align*}
Therefore, $\dual{S_{\omega_0}''(\phi_{\omega_0})w}{w}>0$. 
This completes the proof. 
\end{proof}

\begin{proof}[Proof of Corollary \ref{cor4}]
We verify that $\phi_{\omega_0}$ satisfies the assumptions 
(A1), (A2a) and (A3) of Theorem \ref{thm1}. (A1) follows from (B1). 
Let $\psi=\chi_{1,\omega}$. Then, (A2a) follows from (B2b). 
Finally, (A3) follows from Lemma \ref{lem9}. 
\end{proof}

\section{Examples}\label{sect:examples}

\subsection{Linear Schr\"odinger equation on a bounded interval} \label{ss:1}

We begin with a simple \lq\lq counter-example" 
to emphasize the role of (A3) in Theorem \ref{thm1}. 
We consider the linear Schr\"odinger equation on the interval $(0,\pi)$ 
with zero-Dirichlet boundary conditions 
\begin{equation}\label{eq:8.1}
\left\{\begin{array}{ll}
i\partial_tu-\partial_x^2u=0, &\quad  t\in \R,~ x\in (0,\pi), \\
u(t,0)=u(t,\pi)=0, &\quad t\in \R. 
\end{array}\right.
\end{equation}
Let $H=L^2(0,\pi)$ and $X=H^1_0(0,\pi)$ be real Hilbert spaces with inner products 
$$(u,v)_H=\Re \int_{0}^{\pi}u(x)\overline{v(x)}\,dx, \quad 
(u,v)_X=(\partial_xu,\partial_xv)_H.$$
We define $E(u)=(1/2)\|\partial_xu\|_H^2$ and $Ju=iu$ for $u\in X$. 
$\Trans$ is given by $\Trans (s)u=e^{is}u$ for $u\in X$ and $s\in \R$. 
For $u_0\in X$, the solution $u(t)$ of \eqref{eq:8.1} with $u(0)=u_0$ is expressed as 
$$u(t)=\sum_{n=1}^{\infty}a_n\Trans(n^2t)\varphi_n, \quad 
\varphi_n(x)=\sqrt{\frac{2}{\pi}}\,\sin nx, ~ 
a_n=\int_{0}^{\pi}u_0(x)\varphi_n(x)\,dx.$$
For each $n\in \N$, the bound state $\Trans(n^2t)\varphi_n$ is stable 
in the sense of Definition in Section \ref{sect:form}. 
In particular, we consider the case $n=2$, 
and put $\omega=n^2=4$, $\phi_{\omega}=\varphi_2$ and $\psi=\varphi_1$. 
Then, (A1) and (A2a) are satisfied. 
On the other hand, the inequality \eqref{eq:3.2} holds for $w\in X$ satisfying 
$(\phi_{\omega},w)_H=(J\phi_{\omega},w)_H=(\psi,w)_H=0$ and $(J\psi,w)_H=0$, 
but (A3) does not hold. 
This simple example shows optimality of (A3) in Theorem \ref{thm1}. 

\subsection{NLS with a delta function potential}\label{ss:delta}

We consider a nonlinear Schr\"odinger equation with a delta function potential 
\begin{equation}\label{eq:8.2}
i\partial_tu-\partial_x^2u+\gamma \delta(x)u=|u|^{p-1}u, 
\quad (t,x)\in \R\times \R,
\end{equation}
where $1<p<\infty$, $\gamma\in \R$ and $\delta(x)$ is the delta measure at the origin. 
Although the stability problem of bound states for \eqref{eq:8.2} has been studied 
by many authors (see \cite{FJ,FOO,GHW,LFF}), 
we give some remarks to complement their results. 
For simplicity, we consider the repulsive potential case $\gamma>0$ only. 
As real Hilbert spaces $H$ and $X$, we take $H=L^2(\R)$ and $X=H^1(\R)$
or $H=L^2_{\even}(\R)$ and $X=H^1_{\even}(\R)$. 
We define the inner products of $H$ and $X$ by 
\begin{align*}
&(u,v)_H=\Re \int_{\R}u(x)\overline{v(x)}\,dx, \\
&(u,v)_X=(\partial_xu,\partial_xv)_H+(u,v)_H+\gamma \Re [u(0)\overline{v(0)}]. 
\end{align*}
Note that by the embedding $H^1(\R)\hookrightarrow C_b(\R)$, 
the norm $\|\cdot\|_X$ is equivalent to the usual norm in $H^1(\R)$. 
We define $E:X\to \R$ and $J:X\to X$ by 
$$E(u)=\frac{1}{2}\|\partial_xu\|_{L^2}^2+\frac{\gamma}{2}|u(0)|^2
-\frac{1}{p+1}\|u\|_{L^{p+1}}^{p+1}, \quad Ju=iu$$ 
for $u\in X$. Then, $E\in C^2(X,\R)$ for $1<p<\infty$, $E\in C^3(X,\R)$ if $p>2$, 
and $\Trans$ is given by $\Trans (s)u=e^{is}u$ for $u\in X$ and $s\in \R$. 
Moreover, \eqref{eq:8.2} is written in the form \eqref{eq:2.3}, 
and all the requirements in Section \ref{sect:form} are satisfied. 

For $\omega\in \Omega:=(-\infty,-\gamma^2/4)$, 
\eqref{eq:8.2} has a bound state $e^{i\omega t}\phi_{\omega}(x)$, 
where $\phi_{\omega}\in H^1(\R)$ is a positive solution of 
\begin{equation}\label{eq:8.3}
-\partial_x^2\phi+\gamma \delta(x)\phi-\omega \phi-|\phi|^{p-1}\phi=0, 
\quad x\in \R.
\end{equation}
The positive solution $\phi_{\omega}$ of \eqref{eq:8.3} is given by 
\begin{equation}\label{eq:8.4}
\phi_{\omega}(x)=\left\{\begin{array}{ll}
\varphi_{\omega}(x-b_{\omega}), &\quad x\ge 0, \\
\varphi_{\omega}(x+b_{\omega}), &\quad x<0, 
\end{array}\right.
\end{equation}
where $b_{\omega}=2\tanh^{-1}({\gamma}/{2\sqrt{-\omega}})/[(p-1)\sqrt{-\omega}\,]$, and 
$$\varphi_{\omega}(x)=\left(\frac{-(p+1)\omega}{2}\right)^{1/(p-1)}
\left\{\cosh \left(\frac{(p-1)\sqrt{-\omega}}{2}x \right)\right\}^{-2/(p-1)}$$
is a positive and even solution of 
\begin{equation}\label{eq:8.5}
-\partial_x^2\varphi-\omega \varphi-|\varphi|^{p-1}\varphi=0,\quad x\in \R. 
\end{equation}
Then we see that 
$\omega\mapsto \phi_{\omega}$ is a $C^2$ mapping from $\Omega$ to $X$, and that 
$$R\phi_{\omega}
=-\partial_x^2\phi_{\omega}+\phi_{\omega}+\gamma \delta (x)\phi_{\omega}
=(1+\omega) \phi_{\omega}+|\phi_{\omega}|^{p-1}\phi_{\omega}\in H^1_{\even}(\R).$$
Thus (B1) is satisfied. 
The linearized operator $S_{\omega}''(\phi_{\omega}):X\to X^*$ is given by 
$$\dual{S_{\omega}''(\phi_{\omega})u}{v}
=\dual{L_{\omega}\Re u}{\Re v}+\dual{M_{\omega}\Im u}{\Im v}$$ 
for $u,v\in X$, where 
\begin{align*}
&\dual{L_{\omega}w}{z}=\int_{\R}(\partial_xw\partial_xz-\omega wz
-p\phi_{\omega}(x)^{p-1}wz)\,dx+\gamma w(0)z(0), \\
&\dual{M_{\omega}w}{z}=\int_{\R}(\partial_xw\partial_xz-\omega wz
-\phi_{\omega}(x)^{p-1}wz)\,dx+\gamma w(0)z(0).
\end{align*}
The assumption (B3) is easily verified. It is proved in Lemmas 28 and 29 of \cite{FJ} 
that (B2a) holds for the case $X=H^1_{\even}(\R)$, 
while it is proved in Section 4 of \cite{LFF} that (B2b) holds for the case $X=H^1(\R)$. 
Here, we give a simple proof for the latter fact. 

\begin{lemma}\label{lem10}
$\inf\{\dual{L_{\omega}v}{v}:v\in H^1_{\odd}(\R,\R),~ \|v\|_{L^2}=1\}<0$. 
\end{lemma}

\begin{proof}
Let $s\in (-b_{\omega},\infty)$, and we define 
$$\psi_s(x)=\left\{\begin{array}{ll}
\varphi_{\omega}'(x-b_{\omega}-s), &\quad x>b_{\omega}+s, \\
\varphi_{\omega}'(x+b_{\omega}+s), &\quad x<-b_{\omega}-s, \\
0, &\quad -b_{\omega}-s\le x\le b_{\omega}+s.
\end{array}\right.$$
Then, $\psi_s\in H^1_{\odd}(\R,\R)$ and 
\begin{align*}
f(s):=&\dual{L_{\omega}\psi_s}{\psi_s} \\
=&2\int_{b_{\omega}+s}^{\infty}\{|\varphi_{\omega}''(x-b_{\omega}-s)|^2
-\omega |\varphi_{\omega}'(x-b_{\omega}-s)|^2 \\
&\hspace{20mm} -p\varphi_{\omega}(x-b_{\omega})^{p-1}
|\varphi_{\omega}'(x-b_{\omega}-s)|^2\}\,dx \\
=&\int_{0}^{\infty}\{|\varphi_{\omega}''(y)|^2-\omega |\varphi_{\omega}'(y)|^2
-p\varphi_{\omega}(y+s)^{p-1}|\varphi_{\omega}'(y)|^2\}\,dy.
\end{align*}
Since $\varphi_{\omega}$ is an even solution of \eqref{eq:8.5}, 
we see that $f(0)=0$. Moreover, since 
$$f'(s)=-p(p-1) \int_{0}^{\infty}\varphi_{\omega}(y+s)^{p-2}
\varphi_{\omega}'(y+s)|\varphi_{\omega}'(y)|^2\,dy,$$
we have $f'(0)>0$. Thus, we see that $f(s)<0$ for $s<0$ close to $0$, 
which concludes the lemma. 
\end{proof}

\begin{lemma}\label{lem11}
For each $\omega\in \Omega$, $({\rm B2b})$ holds for $X=H^1(\R)$. 
\end{lemma}

\begin{proof}
By Lemma 31 of \cite{FJ}, the kernel of $S_{\omega}''(\phi_{\omega})$ 
is spanned by $J\phi_{\omega}$, while by Lemma 32 of \cite{FJ}, 
the number of negative eigenvalues of $S_{\omega}''(\phi_{\omega})$ is at most two. 
Moreover, we know that the first eigenvalue $\lambda_{0,\omega}$ 
is negative, and the corresponding eigenfunction $\chi_{0,\omega}\in H^1_{\even}(\R,\R)$. 
By Lemma \ref{lem10}, we have the second eigenvalue $\lambda_{1,\omega}<0$ 
and the corresponding eigenfunction $\chi_{1,\omega}\in H^1_{\odd}(\R,\R)$. 
Since $\phi_{\omega}\in H^1_{\even}(\R,\R)$, we see that 
$(\chi_{0,\omega},\chi_{1,\omega})_H=(\chi_{1,\omega},\phi_{\omega})_H=0$. 
This completes the proof. 
\end{proof}

By the explicit formula \eqref{eq:8.4}, 
we can compute the derivatives of the function $d(\omega)=S_{\omega}(\phi_{\omega})$. 
The following is proved in \cite{FJ}. 
If $1<p\le 3$, then $d''(\omega)>0$ for all $\omega\in \Omega$. 
If $3<p<5$, then there exists $\omega_*\in \Omega$ such that 
$d''(\omega)<0$ for $\omega\in (\omega_*,-\gamma^2/4)$, 
$d''(\omega)>0$ for $\omega\in (-\infty,\omega_*)$, 
$d''(\omega_*)=0$ and $d'''(\omega_*)<0$. 
If $p\ge 5$, then $d''(\omega)<0$ for all $\omega\in \Omega$. 
In particular, for the case where $1<p\le 3$ and $\omega\in \Omega$ 
and for the case where $3<p<5$ and $\omega\in (-\infty,\omega_*)$, 
it follows from Corollary \ref{cor4} that 
$e^{i\omega t}\phi_{\omega}$ is unstable in $X=H^1(\R)$. 
This result is originally due to Theorem 4 of \cite{LFF}. 
However, it seems that the proof in \cite{LFF} is not complete. 
In fact, in Section 4 of \cite{LFF}, 
linear instability of $e^{i\omega t}\phi_{\omega}$ 
is proved by applying the abstract theory of \cite{GSS2}, 
but there is no proof for the assertion that 
linear instability implies (nonlinear) instability 
(see Remark \ref{rem4} in Section \ref{sect:results}). 
Note that, because of the singularity of delta function potential, 
it seems difficult to apply the results available in the literature 
for this problem directly (see \cite{GO} and the references therein), 
and it might be easier to apply Corollary \ref{cor4}. 
While, for the case where $3<p<5$ and $\omega=\omega_*$, 
it follows from Corollary \ref{cor2} that 
$e^{i\omega t}\phi_{\omega}$ is unstable in $X=H^1_{\even}(\R)$, 
which was left open in Remark 7 of \cite{LFF}. 

There are not so many examples such that 
the derivatives of the function $d(\omega)$ can be computed explicitly. 
In \cite{mae1}, one can find other examples 
to which Corollary \ref{cor2} is applicable. 

\subsection{A system of NLS}\label{ss:system}

We consider a system of nonlinear Schr\"odinger equations of the form 
\begin{equation}\label{eq:8.6}
\left\{\begin{array}{l}
i\partial_tu_1-\Delta u_1=|u_1|u_1+\gamma \overline{u_1}u_2, 
\quad (t,x)\in \R\times \R^N, \\
i\partial_tu_2-2\Delta u_2=2|u_2|u_2+\gamma u_1^2, 
\quad (t,x)\in \R\times \R^N, 
\end{array}\right.
\end{equation}
where $N\le 3$ and $\gamma>0$. 
This is a reduced system of a three-component system studied in \cite{CCO1,CCO2}. 

In what follows, we use the vectorial notation $\vec u=(u_1,u_2)$, 
and it is considered to be a column vector. 
We define the inner products of $H=L^2_{\rad}(\R^N)\times L^2_{\rad}(\R^N)$ 
and $X=H^1_{\rad}(\R^N)\times H^1_{\rad}(\R^N)$ by 
\begin{align*}
&(\vec u,\vec v)_H=\Re \int_{\R^N}u_1(x)\overline{v_1(x)}\,dx
+\Re \int_{\R^N}u_2(x)\overline{v_2(x)}\,dx, \\
&(\vec u, \vec v)_X=(\nabla \vec u,\nabla \vec v)_H+(\vec u,\vec v)_H 
\end{align*}
for $\vec u=(u_1,u_2)$ and $\vec v=(v_1,v_2)$. 
We define $J\vec u=(iu_1,2iu_2)$ and 
\begin{align*}
E(\vec u)=\frac{1}{2}\|\nabla u_1\|_{L^2}^2
+\frac{1}{2}\|\nabla u_2\|_{L^2}^2
-\frac{1}{3}\|u_1\|_{L^3}^3-\frac{1}{3}\|u_2\|_{L^3}^3
-\frac{\gamma}{2} \Re \int_{\R^N}u_1^2\overline{u_2}\,dx, 
\end{align*} 
for $\vec u\in X$. Then, \eqref{eq:8.6} is written in the form \eqref{eq:2.3}, 
$\Trans$ is given by $\Trans (s)\vec u=(e^{is}u_1,e^{2is}u_2)$ 
for $\vec u\in X$ and $s\in \R$, 
and all the requirements in Section \ref{sect:form} are satisfied. 
Let $\omega<0$ and let $\varphi_{\omega}\in H^1_{\rad}(\R^N)$ be 
a unique positive radial solution of 
\begin{equation}\label{eq:8.7}
-\Delta \varphi-\omega \varphi-\varphi^2=0, \quad x\in \R^N.
\end{equation}
In the same way as in \cite{CCO1,CCO2}, it is proved that a semi-trivial solution 
$(0,e^{2i\omega t}\varphi_{\omega})$ of \eqref{eq:8.6} 
is stable if $0<\gamma<1$, and unstable if $\gamma>1$. 
Here, we consider instability of bound states bifurcating from 
the semi-trivial solution at $\gamma=1$. For $0<\gamma<1$, we put 
$\vec \phi_{\omega}=(\alpha \varphi_{\omega},\beta \varphi_{\omega})$, where 
$$\alpha=\frac{2-\gamma-\gamma \sqrt{1+2\gamma (\gamma-1)}}{2+\gamma^3}, \quad 
\beta=\frac{1+\gamma^2+\sqrt{1+2\gamma (\gamma-1)}}{2+\gamma^3}.$$
Then, $S_{\omega}'(\vec \phi_{\omega})=0$, and (A1) is satisfied. 
Note that $\alpha$ and $\beta$ are positive constants, 
and satisfy $|\alpha|+\gamma \beta=1$, $\gamma \alpha^2+2|\beta|\beta=2\beta$, 
and $(\alpha,\beta)\to (0,1)$ as $\gamma\to 1$. 
By applying Theorem \ref{thm1}, we show that the bound state 
$\Trans(\omega t)\vec \phi_{\omega}$ is unstable for any $0<\gamma<1$. 
First, the linearized operator $S_{\omega}''(\vec \phi_{\omega})$ is given by 
\begin{equation}\label{eq:8.8}
\dual{S_{\omega}''(\vec \phi_{\omega})\vec u}{\vec u}
=\dual{\mathcal{L}_R \Re \vec u}{\Re \vec u}
+\dual{\mathcal{L}_I \Im \vec u}{\Im \vec u}
\end{equation}
for $\vec u=(u_1,u_2)\in X$, 
where $\Re \vec u=(\Re u_1,\Re u_2)$, $\Im \vec u=(\Im u_1,\Im u_2)$, and 
\begin{align*}
&\mathcal{L}_R=\left[\begin{array}{cc}
-\Delta-\omega & 0 \\
0 & -\Delta-\omega
\end{array}\right]
-\left[\begin{array}{cc}
(2\alpha+\gamma \beta) \varphi_{\omega} & \gamma \alpha \varphi_{\omega} \\
\gamma \alpha \varphi_{\omega} & 2\beta \varphi_{\omega}
\end{array}\right], \\
&\mathcal{L}_I=\left[\begin{array}{cc}
-\Delta-\omega & 0 \\
0 & -\Delta-\omega
\end{array}\right]
-\left[\begin{array}{cc}
(\alpha-\gamma \beta) \varphi_{\omega} & \gamma \alpha \varphi_{\omega} \\
\gamma \alpha \varphi_{\omega} & \beta \varphi_{\omega}
\end{array}\right].
\end{align*}
For $a\in \R$, we define 
$L_av=-\Delta v-\omega v-a\varphi_{\omega} v$ for $v\in H^1_{\rad}(\R^N,\R)$. 
Then, by orthogonal matrices 
$$A=\frac{1}{\sqrt{\alpha^2+\beta^2}}
\left[\begin{array}{cc}
\alpha & \beta \\
-\beta & \alpha 
\end{array}\right], \quad 
B=\frac{1}{\sqrt{\alpha^2+4\beta^2}}
\left[\begin{array}{cc}
\alpha & 2\beta \\
-2\beta & \alpha 
\end{array}\right],$$
$\mathcal{L}_R$ and $\mathcal{L}_I$ are diagonalized as follows: 
\begin{equation}\label{eq:8.9}
\mathcal{L}_R=A^*
\left[\begin{array}{cc}
L_2 & 0 \\
0 & L_{(2-\gamma)\beta}
\end{array}\right]A, \quad 
\mathcal{L}_I=B^*
\left[\begin{array}{cc}
L_1 & 0 \\
0 & L_{(1-2\gamma)\beta}
\end{array}\right]B.
\end{equation}
Moreover, by elementary computations, 
we see that $1<(2-\gamma)\beta<2$ and $(1-2\gamma)\beta<1$ for $0<\gamma<1$. 
Here, we recall some known results on 
the operator $L_a$ defined on $H^1_{\rad}(\R^N,\R)$. 

\begin{lemma}\label{lem12}
Let $N\le 3$ and 
let $\varphi_{\omega}$ be the positive radial solution of \eqref{eq:8.7}. 
\par \noindent $({\rm i})$ \hspace{1mm}
$L_2$ has one negative eigenvalue, $\ker L_2=\{0\}$, 
and there exists a constant $c_1>0$ such that $\dual{L_2v}{v}\ge c_1 \|v\|_{H^1}^2$ 
for all $v\in H^1_{\rad}(\R^N,\R)$ satisfying $(\varphi_{\omega},v)_{L^2}=0$. 
\par \noindent $({\rm ii})$ \hspace{1mm}
$L_1$ is non-negative, $\ker L_1$ is spanned by $\varphi_{\omega}$, 
and there exists $c_2>0$ such that 
$\dual{L_1v}{v} \ge c_2 \|v\|_{H^1}^2$ 
for all $v\in H^1_{\rad}(\R^N,\R)$ satisfying $(\varphi_{\omega},v)_{L^2}=0$. 
\par \noindent $({\rm iii})$ \hspace{1mm}
If $a<1$, then there exists $c_3>0$ such that 
$\dual{L_a v}{v}\ge c_3 \|v\|_{H^1}^2$ for all $v\in H^1_{\rad}(\R^N,\R)$. 
\par \noindent $({\rm iv})$ \hspace{1mm}
If $1<a<2$, then $\dual{L_a \varphi_{\omega}}{\varphi_{\omega}}<0$, 
and there exists $c_4>0$ such that $\dual{L_a v}{v}\ge c_4 \|v\|_{H^1}^2$ 
for all $v\in H^1_{\rad}(\R^N,\R)$ satisfying $(\varphi_{\omega},v)_{L^2}=0$. 
\end{lemma}

\begin{proof}
The parts (i) and (ii) are well-known (see \cite{wei1}). 
Note that the quadratic nonlinearity in \eqref{eq:8.7} 
is $L^2$-subcritical if and only if $N\le 3$, 
and that the assumption $N\le 3$ is essential for (i). 
The parts (iii) and (iv) follow from (i) and (ii) immediately. 
\end{proof}

We put $\vec \xi=(-\beta \varphi_{\omega},\alpha \varphi_{\omega})$ 
and $\vec \psi=\vec \xi/\|\vec \xi\|_H$. 
Then, $A\vec \psi=(0,\varphi_{\omega})/\|\varphi_{\omega}\|_{L^2}$. 
By Lemma \ref{lem12} (iv), we have 
$$\dual{S_{\omega}''(\vec \phi_{\omega})\vec \psi}{\vec \psi}
=\dual{\mathcal{L}_R \vec \psi}{\vec \psi}
=\dual{L_{(2-\gamma)\beta} \varphi_{\omega}}{\varphi_{\omega}}
/\|\varphi_{\omega}\|_{L^2}^2<0,$$ 
and (A2a) is satisfied. Next, we show two lemmas to prove (A3). 

\begin{lemma}\label{lem13}
There exists a constant $k_1>0$ such that 
$\dual{\mathcal{L}_R \vec v}{\vec v}\ge k_1 \|\vec v\|_{X}^2$ 
for all $\vec v\in H^1_{\rad}(\R^N,\R)^2$ satisfying 
$(\vec \phi_{\omega},\vec v)_H=0$ and $(\vec \xi,\vec v)_H=0$. 
\end{lemma}

\begin{proof}
By \eqref{eq:8.9}, we have $\dual{\mathcal{L}_R \vec v}{\vec v}
=\dual{L_2w_1}{w_1}+\dual{L_{(2-\gamma)\beta}w_2}{w_2}$, 
where $\vec w=A\vec v$. Since $(\varphi_{\omega},w_1)_{L^2}
=(\vec \phi_{\omega},\vec v)_H/\sqrt{\alpha^2+\beta^2}=0$, 
Lemma \ref{lem12} (i) implies 
$\dual{L_2w_1}{w_1}\ge c_1 \|w_1\|_{H^1}^2$. Moreover, since 
$(\varphi_{\omega},w_2)_{L^2}=(\vec \xi,\vec v)_H/\sqrt{\alpha^2+\beta^2}=0$ 
and $1<(2-\gamma)\beta<2$, Lemma \ref{lem12} (iv) implies 
$\dual{L_{(2-\gamma)\beta}w_2}{w_2}\ge c_4 \|w_2\|_{H^1}^2$. 
Since $\|\vec w\|_X=\|\vec v\|_X$, this completes the proof. 
\end{proof}

\begin{lemma}\label{lem14}
There exists a constant $k_2>0$ such that 
$\dual{\mathcal{L}_I \vec v}{\vec v}\ge k_2 \|\vec v\|_{X}^2$ 
for all $\vec v\in H^1_{\rad}(\R^N,\R)^2$ satisfying 
$(\vec \eta,\vec v)_H=0$, where 
$\vec \eta=(\alpha \varphi_{\omega},2\beta \varphi_{\omega})$. 
\end{lemma}

\begin{proof}
By \eqref{eq:8.9}, we have $\dual{\mathcal{L}_I \vec v}{\vec v}
=\dual{L_1w_1}{w_1}+\dual{L_{(1-2\gamma) \beta}w_2}{w_2}$, 
where $\vec w=B\vec v$. 
Since $(\varphi_{\omega},w_1)_{L^2}=(\vec \eta,\vec v)_H/\sqrt{\alpha^2+4\beta^2}=0$, 
Lemma \ref{lem12} (ii) implies $\dual{L_1\tilde v_1}{w_1}\ge c_2 \|w_1\|_{H^1}^2$. 
Moreover, since $(1-2\gamma) \beta<1$, Lemma \ref{lem12} (iii) implies 
$\dual{L_{(1-2\gamma) \beta}w_2}{w_2}\ge c_3 \|w_2\|_{H^1}^2$. 
This completes the proof. 
\end{proof}

We verify (A3). Let $\vec w\in  X$ satisfy $(\vec \phi_{\omega},\vec w)_H
=(J\vec \phi_{\omega},\vec w)_H=(\vec \psi,\vec w)_H=0$. 
Since $(\vec \phi_{\omega},\Re \vec w)_H=(\vec \phi_{\omega},\vec w)_H=0$ 
and $(\vec \xi,\Re \vec w)_H=\|\vec \xi\|_H(\vec \psi,\vec w)_H=0$, 
it follows from Lemma \ref{lem13} that 
$\dual{\mathcal{L}_R \Re \vec w}{\Re \vec w}\ge k_1\|\Re \vec w\|_{X}^2$. 
While, since $(\vec \eta,\Im \vec w)_H=-(J\phi_{\omega},\vec w)_H=0$, 
Lemma \ref{lem14} implies 
$\dual{\mathcal{L}_I \Im \vec w}{\Im \vec w}\ge k_2\|\Im \vec w\|_{X}^2$. 
Thus, by \eqref{eq:8.8}, we see that (A3) is satisfied. 
In conclusion, it follows from Theorem \ref{thm1} that the bound state 
$\Trans(\omega t)\vec \phi_{\omega}$ is unstable for any $0<\gamma<1$. 

Finally, we consider instability of semi-trivial solution 
$\Trans(\omega t)(0,\varphi_{\omega})$ at the bifurcation point $\gamma=1$. 
In this case, we have $\mathcal{L}_R\vec v=(L_1v_1,L_2v_2)$ and 
$\mathcal{L}_I\vec v=(L_{-1}v_1,L_1v_2)$ 
for $\vec v=(v_1,v_2)\in H^1_{\rad}(\R^N,\R)^2$, 
the kernel of $S_{\omega}''(0,\varphi_{\omega})$ is spanned by 
$J(0,\varphi_{\omega})$ and $(\varphi_{\omega},0)$, 
and (A3) holds with $\psi=(\varphi_{\omega},0)/\|\varphi_{\omega}\|_{L^2}$. 
Since $E\notin C^3(X,\R)$, Corollary \ref{cor1} is not applicable to this problem directly. 
However, by modifying the proof of Theorem \ref{thm2}, 
it is proved that $\Trans(\omega t)(0,\varphi_{\omega})$ is unstable for the case $\gamma=1$. 
The detail will be discussed in a forthcoming paper \cite{CO}. 

\vspace{3mm}\noindent 
\textbf{Acknowledgements.}
This work was supported by JSPS Excellent Young Researchers Overseas Visit Program 
and by JSPS KAKENHI (21540163). 
The author is grateful to Mathieu Colin and Masaya Maeda for useful discussions.


\begin{thebibliography}{99}

\bibitem{CCO1}M. Colin, T. Colin and M. Ohta, 
Stability of solitary waves for a system of 
nonlinear Schr\"odinger equations with three wave interaction, 
Ann. Inst. H. Poincar\'{e}, Anal. Non Lin\'{e}aire 
\textbf{26} (2009) 2211--2226. 

\bibitem{CCO2}M. Colin, T. Colin and M. Ohta, 
Instability of standing waves for a system of 
nonlinear Schr\"odinger equations with three-wave interaction, 
Funkcial. Ekvac. \textbf{52} (2009) 371--380. 

\bibitem{CO}M. Colin and M. Ohta, 
Bifurcation from semi-trivial standing waves and ground states 
for a system of nonlinear Schr\"odinger equations, 
preprint, arXiv:1102.1545. 

\bibitem{CP}A. Comech and D. Pelinovsky, 
Purely nonlinear instability of standing waves with minimal energy, 
Comm. Pure Appl. Math. \textbf{56} (2003) 1565--1607. 

\bibitem{ES}M. Esteban and W. Strauss, 
Nonlinear bound states outside an insulated sphere, 
Comm. Partial Differential Equations \textbf{19} (1994) 177--197.

\bibitem{FJ}R. Fukuizumi and L. Jeanjean, 
Stability of standing waves for a nonlinear Schr\"odinger equation 
with a repulsive Dirac delta potential, 
Discrete Contin. Dyn. Syst. \textbf{21} (2008) 121--136. 

\bibitem{FOO}R. Fukuizumi, M. Ohta and T. Ozawa, 
Nonlinear Schr\"odinger equation with a point defect, 
Ann. Inst. H. Poincar\'e, Anal. Non Lin\'eaire 
\textbf{25} (2008) 837--845.

\bibitem{GO}V. Georgiev and M. Ohta, 
Nonlinear instability of linearly unstable standing waves for 
nonlinear Schr\"{o}dinger equations, 
preprint, arXiv:1009.5184. 

\bibitem{GHW}R. H. Goodman, P. J. Holmes, and M. I. Weinstein, 
Strong NLS soliton-defect interactions, 
Phys. D \textbf{192} (2004) 215--248. 

\bibitem{gri}M. Grillakis, 
Linearized instability for nonlinear Schr\"odinger and Klein-Gordon equations,
Comm. Pure Appl. Math. \textbf{41} (1988) 747--774. 

\bibitem{GSS1}M. Grillakis, J. Shatah and W. Strauss,
Stability theory of solitary waves in the presence of symmetry I,
J. Funct. Anal. \textbf{74} (1987) 160--197.

\bibitem{GSS2}M. Grillakis, J. Shatah and W. Strauss,
Stability theory of solitary waves in the presence of symmetry II,
J. Funct. Anal. \textbf{94} (1990) 308--348. 

\bibitem{jon}C. K. R. T. Jones,
An instability mechanism for radially symmetric standing waves 
of a nonlinear Schr\"odinger equation, 
J. Differential Equations \textbf{71} (1988) 34--62.

\bibitem{KKSW}E. W. Kirr, P. G. Kevrekidis, E. Shlizerman and M. I. Weinstein, 
Symmetry-breaking bifurcation in nonlinear Schr\"odinger/Gross-Pitaevskii equations, 
SIAM J. Math. Anal. \textbf{40} (2008) 566--604. 

\bibitem{LFF}S. Le Coz, R. Fukuizumi, G. Fibich, B. Ksherim and Y. Sivan, 
Instability of bound states of a nonlinear Schr\"odinger equation with a Dirac potential, 
Phys. D \textbf{237} (2008) 1103--1128.

\bibitem{LL}E. H. Lieb and M. Loss, 
Analysis, second edition, 
Grad. Stud. Math., vol. 14, Amer. Math. Soc., Providence, RI, 2001. 

\bibitem{mae1}M. Maeda, 
Stability and instability of standing waves for 1-dimensional nonlinear 
Schr\"odinger equation with multiple-power nonlinearity, 
Kodai Math. J. \textbf{31} (2008) 263--271. 

\bibitem{mae2}M. Maeda,
Instability of bound states of nonlinear Schr\"{o}dinger equations
with Morse index equal to two,
Nonlinear Anal. \textbf{72} (2010) 2100--2113.

\bibitem{OT}M. Ohta and G. Todorova, 
Strong instability of standing waves for the nonlinear Klein-Gordon equation 
and the Klein-Gordon-Zakharov system, 
SIAM J. Math. Anal. \textbf{38} (2007) 1912--1931. 

\bibitem{SS1}J. Shatah and W. Strauss,
Instability of nonlinear bound states,
Comm. Math. Phys. \textbf{100} (1985) 173--190.

\bibitem{stu}C. A. Stuart, 
Lectures on the orbital stability of standing waves and application to 
the nonlinear Schr\"odinger equations, 
Milan J. Math. \textbf{76} (2008) 329--399. 

\bibitem{wei1}M. I. Weinstein, 
Modulational stability of ground states 
of nonlinear Schr\"odinger equations, 
SIAM J. Math. Anal. \textbf{16} (1985) 472--491. 

\bibitem{wei2}M. I. Weinstein, 
Lyapunov stability of ground states of nonlinear dispersive evolution equations, 
Comm. Pure Appl. Math. \textbf{39} (1986) 51--68. 

\end{thebibliography}
\end{document}